\documentclass[12pt]{amsart}
\usepackage{amsfonts, amsmath, amsthm, amssymb, stmaryrd}
\usepackage{textcomp}
\usepackage{bbm}
\usepackage{enumitem, tabto}
\usepackage{xcolor}
\usepackage{url}
\usepackage{graphicx}
\usepackage{fancybox}
\usepackage{awesomebox}
\usepackage{tikz}
\usepackage{pgfplots}
\usepackage[Glenn]{fncychap}
\usepackage{fancyhdr}
\usepackage{lastpage}
\usepackage{import}

\usepackage{diagbox}
\usepackage{multirow} 
\usepackage{tabularx}

\usepackage{etoolbox}

\usepackage{hyperref}

\usepackage{blindtext}
\usepackage{multicol}
\usepackage{wrapfig}

\usepackage{tkz-tab}

\usepackage{pdfpages}

\usepackage{comment}

\newcommand{\N}{\mathbb{N}}

\newcommand{\C}{\mathbb{C}}
\newcommand{\D}{\mathbb{D}}

\newcommand{\T}{\mathbb{T}}
\newcommand{\Lcont}{\mathcal{L}}

\newcommand{\X}{\mathcal{X}}
\newcommand{\I}{\mathcal{I}}
\newcommand{\A}{\mathcal{A}}

\newcommand{\norm}[1]{\left\Vert #1\right\Vert}
\newcommand{\abs}[1]{\left\lvert #1 \right\rvert}
\newcommand{\scal}[2]{\ensuremath{\left\langle #1|#2 \right\rangle}\xspace}

\newcommand{\Hol}{\text{Hol}}

\usepackage{indentfirst}
\makeatletter
\@addtoreset{equation}{section}
\makeatother
\theoremstyle{plain}
\newtheorem{statement}{}[section]
\newtheorem{theo}[statement]{Theorem}
\newtheorem{prop}[statement]{Proposition}

\newtheorem{coro}[statement]{Corollary}
\newtheorem{lem}[statement]{Lemma}

\newtheorem{quest}[statement]{Question}

\newtheorem{theospec}{Theorem}
\newtheorem{corospec}[theospec]{Corollary}

\theoremstyle{definition}

\newtheorem{rem}[statement]{Remark}

\title{Cyclicity of the shift operator through Bezout identities}
\author{Emmanuel Fricain and Romain Lebreton}
\address{Laboratoire Paul Painlevé, Université de Lille, 59655 Villeneuve d'Ascq Cédex}
\email{emmanuel.fricain@univ-lille.fr}
\email{romain.lebreton@univ-lille.fr}

\thanks{The authors were supported by Labex CEMPI (ANR-11-LABX-0007-01)}

\begin{document}

\maketitle

\begin{abstract}
In this paper, we study the cyclicity of the shift operator $S$ acting on a Banach space $\X$ of analytic functions on the open unit disc $\D$. We develop a general framework where a method based on a corona theorem can be used to show that if $f,g\in\X$ satisfy $|g(z)|\leq |f(z)|$, for every $z\in\D$, and if $g$ is cyclic, then $f$ is cyclic. We also give sufficient conditions for cyclicity in this context. This enable us to recapture some recent results obtained in de Branges-Rovnayk spaces, in  Besov--Dirichlet spaces and in weighted Dirichlet type spaces. 
\end{abstract}

\section{Introduction}

If $T$ is a bounded linear operator on a Banach space $\X$, then $T$ is said to be \textit{cyclic} if there exists a vector $x \in\X$ such that the orbit of $x$ under $T$, defined by $\{T^n x : n \ge 0 \}$, generates a dense subspace in $\X$. Such a vector, if it exists, is called a \textit{cyclic vector} for $T$. The characterization of cyclic vectors of a given operator is a challenging question which has connections with the famous invariant subspace problem. 

The cyclicity problem with respect to the (forward) shift operator
$S:f(z)\longmapsto zf(z)$ has been completely solved by A. Beurling \cite{Beurling} in the context of the Hardy space $H^2$ of the open unit disc $\D$: a function $f$ in $H^2$ is cyclic for $S$ if and only if $f$ is an outer function, in the sense that $f$ can be written as 
\[
f(z)=c\exp\left(\int_\T \frac{\xi+z}{\xi-z}\log(\varphi(\xi))\,dm(\xi)\right),\qquad |z|<1,
\]
where $|c|=1$, $m$ is normalized Lebesgue measure on the unit circle $\T$ and $\varphi$ is a nonnegative function in $L^2(\T)$ such that $\log(\varphi)\in L^1(\T)$.  

A similar question can be stated (and has been studied) in various Banach spaces of analytic functions where the shift operator acts boundedly, e.g. in Bergman or Dirichlet spaces. However, the situation in Hardy space is unique in the sense that in most other spaces, there are no known characterizations despite numerous efforts by many mathematicians. Cyclic vectors in the Dirichlet space were initially studied by L. Carleson \cite{MR50011}, and later by L. Brown and A. Shields \cite{BrownShields}. In this last paper, the authors conjectured that a function $f$ in the Dirichlet space $\mathcal{D}$ is cyclic for the shift operator if and only if $f$ is outer and its boundary zero set is of logarithmic capacity zero. This conjecture is still open despite significant progress \cite{CyclDirichlet}.  Brown--Shields also posed the following question about cyclic vectors for the shift acting on a  general Banach space $\X$ of analytic functions on $\D$ (with some standard properties):
\begin{quest}[\cite{BrownShields}]\label{question-BS}
If $f,g\in\X$ satisfy $|g(z)|\leq |f(z)|$ for every $z\in\D$ and if $g$ is cyclic for the shift, then must $f$ be cyclic for the shift? 
\end{quest}
They proved that if the algebra of multipliers of $\X$ coincides with $H^\infty$, the algebra of analytic and bounded functions on $\D$, then the answer is positive. They also showed that this is the case when $\X$ coincides with the Dirichlet space $\mathcal D$. Finally, if $f\in\mathcal D_2$ (the weighted Hardy space on $\D$ with weight $(n+1)^2$) and if $f$ has at most countably many zeros on $\T$, then $f$ is cyclic for the shift in $\mathcal D$. Note also the reference \cite{RichterSundbergDirichlet} where the authors gave similar results in the context of Dirichlet space $\mathcal D(\mu)$ for $\mu$ a nonnegative finite Borel measure on $\T$. 

In \cite{MR2979822}, Carleson's corona theorem is used to get two results on cyclicity for singular inner functions in weighted Bergman type spaces on the open unit disc. In \cite{TheseVectCycl,Egueh-Kellay-Zarrabi}, this method based on corona theorem is pursued to get a positive answer to Question \ref{question-BS} in the context of Besov-Dirichlet spaces $\mathcal D_\alpha^p$. See also \cite{MR2056436,MR2679531,MR769756}.

Inspired by \cite{TheseVectCycl,Egueh-Kellay-Zarrabi}, the aim of this paper is to develop a general framework where the method based on a corona theorem can be applied to get a positive answer to Question \ref{question-BS}. More precisely, consider a Banach space $\mathcal X$ of analytic functions on $\D$ satisfying standard assumptions, and a Banach algebra $\mathcal A$ which is contained in the mutilpliers algebra of $\X$ and which satisfies a corona theorem with some control on the solutions. The main results of this paper are the following. See Subsections \ref{subsection-general-context} and \ref{subsection-operator-tools} for the precise statement of assumptions (H1) to (H8). 
\begin{theospec}\label{theorem-1}
Let $\X$ satisfying (H1) to (H3) and let $\mathcal A$ satisfying  (H4) to  (H6). Then there exists $N\in\mathbb N^*$ such that for every $f,g\in\mathcal A$ satisfying $|g(z)|\leq |f(z)|$ for every $z\in\mathbb D$, we have
\[
[g^N]_\X\subset [f]_\X.
\]
\end{theospec}
Here for $f\in\X$, we denote by $[f]_\X$ the smallest $S$-invariant subspace containing $f$, that is 
\[
[f]_\X=\overline{\{pf:p\in\mathcal P\}}^\X,
\]
where $\mathcal P$ is the set of polynomials and $\overline{\{\cdots\}}^\X$ denotes the closure of $\{\cdots\}$ in $\X$. As a corollary, we will get a (partial) positive answer to Brown--Shields question in our context.
\begin{corospec}\label{corollaire-1}
Let $\X$ satisfying  (H1) to  (H3) and let $\mathcal A$ satisfying (H4) to (H6). Let $f,g\in\mathcal A$ such that for every $z\in\D$, we have $|g(z)|\leq |f(z)|$. Suppose $g$ is cyclic  for $S$ in $\X$. Then $f$ is cyclic for $S$ in $\X$. 
\end{corospec}

Finally, combining this approach based on a corona theorem and a tauberian result of Atzmon, we also prove the following. The disc algebra consisting of holomorphic functions on $\D$ which are continuous on the closed unit disc $\overline{\D}$ is denoted by $A(\D)$, and for $f\in A(\D)$, we denote by $\mathcal Z(f)$ the boundary zero sets of $f$, that is 
\[
\mathcal Z(f)=\{\zeta\in\T:f(\zeta)=0\}.
\]
\begin{theospec}\label{theorem-2}
Let $\X$ satisfying (H1) to  (H3) and let $\mathcal A$ satisfying (H4) to (H8). Assume that there exists $\zeta_0\in\T$ such that $z-\zeta_0$ is cyclic for $S$ in $\X$. Let $f\in\mathcal A\cap A(\D)$ be an outer function such that $\mathcal Z(f)=\{\zeta_0\}$. Then $f$ is cyclic for $S$ in $\X$. 
\end{theospec}

The paper is organized as follows. In section \ref{context}, we introduce the general framework of our results and some standard preliminaries. Section \ref{cyclicity} contains the proofs of Theorem~\ref{theorem-1} and Corollary~\ref{corollaire-1}. In Section \ref{Section4}, we present the proof of Theorem~\ref{theorem-2}. Finally, in the last section, we present some concrete applications and show that our results enable us to recapture some recent results obtained for $\X=\mathcal H(b)$, the de Branges-Rovnyak space with a rational (not inner) function in the closed unit ball of $H^\infty$, for $\X=\mathcal D_\alpha^p$  the Besov--Dirichlet space with $p>1$ and $\alpha+1\leq p\leq \alpha+2$, and for $\X=\mathcal D(\mu)$, the weighted Dirichlet type space where $\mu$ is a finite positive Borel measure on the closed unit disc $\overline{\D}$. 
\\

We will use sometimes in the paper the notation $A \lesssim B$ meaning that there is an absolute positive constant $C$ such that $A \leq CB$.

\section{Presentation of the context and preliminaries}\label{context}

\subsection{Our general framework}\label{subsection-general-context}
We begin by presenting the assumptions that our Banach space $\X$ must satisfy. For $n\in\mathbb N$, we define $\chi_n(z)=z^n$, $z\in\D$. We assume that $\X$ is a Banach space of analytic functions on $\D$ satisfying the following standard conditions:
\begin{enumerate}
\item[(H1)] For every $\lambda\in\D$, the evaluation map  $E_\lambda : f \in \X \longmapsto f(\lambda) \in \C$ is continuous.
\item[(H2)] For every $f\in\X$, we have  $\chi_1 f \in \X$.
\item[(H3)] The set of polynomials $\mathcal P= \bigvee(\chi_n : n \ge 0)$ is dense in $\X$. 
\end{enumerate}
The assumption (H1) means that, for every $\lambda\in\D$, $E_\lambda$ belongs to $\X^*$, the dual space of $\X$. In particular, it implies that convergence in $\X$ implies pointwise convergence on $\D$. We shall denote by $\langle f|\varphi \rangle=\varphi(f)$ the duality bracket between $\varphi\in\X^*$ and $f\in\X$.

The assumption (H2) means that $\chi_1$ belongs to $\mathfrak{M}(\X)$, the algebra of multipliers of $\X$ defined by $$\mathfrak{M}(\X) := \left \{\varphi \in \X : \varphi f \in \X,\,\forall f \in \X \right \}.$$
Using closed graph theorem and (H1) it is easy to see that if $\varphi\in\mathfrak{M}(\X)$, then the multiplication operator by $\varphi$, $M_\varphi : f \in \X \longmapsto \varphi f \in \X$, is bounded on $\X$. In particular, it follows from (H2) that the shift operator acting on $\X$ defined by 
\[ 
S : f \in \X \longmapsto \chi_1 f \in \X
\] 
is bounded. Moreover, if we set $\norm{\varphi}_{\mathfrak{M}(\X)} = \norm{M_\varphi}_{\Lcont(\X)}$, where $\|\cdot\|_{\Lcont(\X)}$ denotes the (operator) norm on $\Lcont(\X)$, the space of linear and bounded operators on $\X$, then it is well-known that $\mathfrak{M}(\X)$ is a Banach algebra. Note also that by (H2), we have $\mathcal{P}\subset\mathfrak{M}(\X)$.\\

We introduce now a (commutative and unital) Banach algebra $\A$ satisfying the following conditions:
\begin{enumerate}
\item[(H4)] $\A \subset \mathfrak{M}(\X)$.
\item[(H5)] For every $\lambda \in \D$ the evaluation map $f \in \A \longmapsto f(\lambda) \in \C$ is continuous.
\item[(H6)] There exists $C>0$ and $A\geq 1$ such that for every $f_1,f_2\in\A$ satisfying 
\[
0 < \delta \le \abs{f_1} + \abs{f_2} \le 1\qquad \text{~on~} \D,
\]
there exists $g_1,g_2\in\A$ such that $f_1g_1+f_2g_2\equiv 1$ on $\D$  and 
\[
\|g_1\|_\A,\,\|g_2\|_\A\leq \frac{C}{\delta^A}.
\]
\end{enumerate}
  
The assumption (H6) means that the Banach algebra $\A$ satisfies a Corona Theorem with a control on the solutions. It is known to be true in the algebra $H^\infty$, with any $A>2$. See \cite{MR141789,MR629839,Uch}. V. Tolokonnikov \cite{Tolokonnikov} also proved that (H6) is satisfied for $\A=\mathcal D_\alpha^p\cap A(\D)$ with $1<p\leq \alpha+2$ and $A\geq 4$. Let us also mention  the recent paper \cite{MR4363747} of Shuaibing Luo who proved that (H6) holds for $\A=\mathfrak M(\mathcal D(\mu))$ with $\mu$  a finite positive Borel measure on the closed unit disc $\overline{\D}$ and $A\geq 4$, and the paper \cite{CoronaHbMult} who studied the case of the  algebra of multipliers of some de Branges--Rovnyak spaces.

\subsection{Some technical preliminaries}

We now present some simple consequences of our assumptions. Let us first remind a standard property for the multiplier algebra $\mathfrak{M}(\X)$. For completeness, we give a proof.
\begin{lem}\label{multpXHinftyX}
Let $\X$ be a Banach space of analytic functions on $\D$ satisfying (H1) to (H3). 
\begin{enumerate}
\item[$(1)$] We have $\mathfrak{M}(\X) \subset H^\infty \cap \X$ and there exists $c_1> 0$ such that for every $f \in \mathfrak{M}(\X)$, we have $$\norm{f}_\X + \norm{f}_\infty \le c_1 \norm{f}_{\mathfrak{M}(\X)}.$$
\item[$(2)$]  If we assume furthermore that $\mathfrak{M}(\X)=H^\infty \cap \X$, then there exists a constant $c_2>0$ such that for every $f\in \mathfrak{M}(\X)$, we have
\[
c_2\|f\|_{\mathfrak{M}(\X)}\leq \|f\|_\infty+\|f\|_\X.
\]
\end{enumerate}
\end{lem}

\begin{proof}
$(1)$: Let $f \in \mathfrak{M}(\X)$. Since $\chi_0 = 1 \in \X$, $f = f \chi_0 \in \X$, whence $$\norm{f}_\X = \norm{f \chi_0}_\X \le \norm{f}_{\mathfrak{M}(\X)} \norm{\chi_0}_\X.$$

 Note that $M_f^* : \X^* \rightarrow \X^*$ is well-defined and, from (H1), $E_\lambda \in \X^*$ for every $\lambda \in \D$. Then, for every $g \in \X$, we have $$\scal{g}{M_f^* E_\lambda} = \scal{fg}{E_\lambda} = f(\lambda)g(\lambda) = f(\lambda) \scal{g}{E_\lambda} = \scal{g}{\overline{f(\lambda)}E_\lambda}$$ and we deduce that $M_f^*E_\lambda = \overline{f(\lambda)}E_\lambda$. Therefore $$\abs{f(\lambda)}\norm{E_\lambda}_{\X^*} \le \norm{M_f}_{\Lcont(\X)} \norm{E_\lambda}_{\X^*}.$$
Remark that $\scal{\chi_0}{E_\lambda} = 1$, hence $E_\lambda \neq 0$ and dividing the last inequality by $\|E_\lambda\|_{\X^*}$ gives 
    \[
    \abs{f(\lambda)} \le \norm{M_f}_{\Lcont(\X)}.
    \] 
    
In other words, $f$ is bounded on $\D$ and $\norm{f}_\infty \le \norm{f}_{\mathfrak{M}(\X)}$. Moreover, since $f \in \X \subset \Hol(\D)$, $f \in H^\infty$ and 
    \begin{equation*}
        \norm{f}_\X + \norm{f}_\infty \le \left (\norm{\chi_0}_\X + 1 \right )\norm{f}_{\mathfrak{M}(\X)}.
    \end{equation*} 

$(2)$: Equip $H^\infty\cap\X$ with the following norm
\[
\|f\|_{\infty,\X}=\|f\|_\infty+\|f\|_\X,\qquad f\in H^\infty\cap\X.
\]    
It is not difficult to see that $(H^\infty\cap\X,\|\cdot\|_{\infty,\X})$ is a Banach space. Consider now the identity map
\[
\begin{array}{rccl}
i:&\mathfrak{M}(\X)&\longrightarrow & H^\infty\cap\X \\
&f&\longmapsto & f
\end{array}.
\]    
 According to $(1)$, this map is continuous and the theorem of isomorphism of Banach gives the result.  \qedhere
\end{proof}

According to the closed graph theorem applied to the canonical injection $i : f \in \A \longmapsto f \in \mathfrak{M}(\X)$, the assumption (H5) and Lemma \ref{multpXHinftyX}, it is easy to see that there exists $c_3>0$ such that for every $f \in \A$, we have 
\begin{equation}\label{eq:Mx-et-A}
\norm{f}_{\mathfrak{M}(\X)} \le c_3 \norm{f}_\A,
\end{equation}
and then there exists $c_4>0$ such that for every $f \in \A$, we have 
\begin{equation}\label{eq:X-et-A}
\norm{f}_{\X} \le c_4 \norm{f}_\A.
\end{equation}

The purpose of this paper is to study the cyclicity of $S$ on $\X$. Let us remind the following notation 
\[
[f]_\X=\overline{\{pf:p\in\mathcal P\}}^\X, \quad\mbox{for }f\in\X,
\]
where $\overline{\{\cdots\}}^\X$ denotes the closure of $\{\cdots\}$ in $\X$. It is clear that $f$ is a cyclic vector of $S$ in $\X$ if and only if 
\[
[f]_\X=\X.
\]
\begin{rem}\label{remark-2.2-terefd}
It is standard that $f$ is cyclic for $S$ in $\X$ if and only if $1 \in \left [f \right]_\X$. Indeed, since $\chi_0 = 1 \in \X$ by (H3), if $f$ is cyclic for $S$ in $\X$, then $1 \in [f]_\X$. Conversely, let $g \in \X$ and $\varepsilon > 0$. From the assumption (H3), there exists $q \in \mathcal{P}$ such that $\norm{q - g}_\X \le \frac{\varepsilon}{2}$. Moreover, if $1 \in [f]_\X$, then there exists $p \in \mathcal{P}$ such that $\norm{p f - 1}_\X \le \frac{\varepsilon}{2\norm{q}_{\mathfrak{M}(\X)}}$. Thus $pq \in \mathcal{P}$ and 
\[
\norm{pq f - g}_\X \le \norm{pqf - q}_\X + \norm{q - g}_\X \le \norm{q}_{\mathfrak{M}(\X)} \norm{pf - 1}_\X + \frac{\varepsilon}{2} \le \varepsilon,
\]
which gives the cyclicity of $f$.
\end{rem}
\begin{rem}
It also follows from (H1) and Remark~\ref{remark-2.2-terefd} that if $f$ is cyclic for $S$ in $\X$, then for every $\lambda\in\D$, $f(\lambda)\neq 0$. Indeed, there should exists a sequence of polynomials $(p_n)_n$ such that $p_nf\to 1$ as $n\to\infty$. According to (H1), for every $\lambda\in\D$, we thus have $p_n(\lambda)f(\lambda)\to 1$ as $n\to\infty$. This forces $f(\lambda)$ to  be non zero. 
\end{rem}
For $f \in \A$, we denote by $\I_f$ the closed ideal of $\A$ generated by $f$, that is 
\begin{equation}\label{def-ideal}
\I_f := \overline{\left \{ gf : g \in \A \right \}}^{\A}.\
\end{equation}

\begin{lem}\label{IdealfIncluCyclf}
    Let $f \in \A$. Then $\I_f \subset [f]_\X$.
\end{lem}

\begin{proof}
    We show first that for every $g \in \A$, we have $gf \in [f]_{\X}$. It follows from (H4) and Lemma~\ref{multpXHinftyX}  that  $g\in \X$. Now, according to (H3), there exists a sequence of polynomials $(p_n)_{n \ge 0}$ such that $\norm{p_n - g}_\X \to 0$ as $n\to\infty$. Since $f\in\A\subset\mathfrak{M}(\X)$, we obtain that $\norm{p_n f - gf}_\X\to 0$ as $n\to\infty$, and thus $gf \in [f]_\X$.
    
    Consider now $\varphi \in \I_f$. By definition, there exists a sequence $(\varphi_n)_{n \ge 0}$ of $\A$ such that $\norm{\varphi_n f - \varphi}_\A \to 0$ as $n\to\infty$. Then it follows from \eqref{eq:X-et-A} that   $\|\varphi_n f-\varphi\|_\X\to 0$ as $n\to \infty$. Finally, since $\varphi_n f$ belongs to $[f]_\X$, that is closed in $\X$, we conclude that  $\varphi \in [f]_\X$ and thus $\I_f \subset [f]_\X$.
\end{proof}
 
\begin{lem}\label{cyclProdXA}
    Let $f \in \X$ and $\varphi \in \A$ be two cyclic vectors for $S$ in $\X$. Then $f \varphi$ is cyclic for $S$ in $\X$. 
\end{lem}

\begin{proof}
    Let $\varepsilon > 0$. Since $\varphi$ is cyclic for $S$ in $\X$, there exists $p \in \mathcal{P}$ such that $$\norm{p\varphi - 1}_\X \le \frac{\varepsilon}{2}.$$ Moreover, from (H2) and (H4), $\varphi \in \A \subset \mathfrak{M}(\X)$ and $p \in \mathcal{P} \subset \mathfrak{M}(\X) $, whence $p\varphi \in \mathfrak{M}(\X)$. Now, since $f$ is also cyclic for $S$ in $\X$, there exists $q \in \mathcal{P}$ such that $$\norm{q f - 1}_\X \le \frac{\varepsilon}{2 \norm{p \varphi}_{\mathfrak{M}(\X)}}.$$ 
   Thus
    \begin{align*}
        \norm{f \varphi pq - 1}_\X & \le \norm{f \varphi pq - p\varphi}_\X + \norm{p\varphi - 1}_\X \\ & \le \norm{\varphi p(q f - 1)}_\X + \frac{\varepsilon}{2} \\
       & \le \norm{p \varphi}_{\mathfrak{M}(\X)}\norm{qf - 1}_\X + \frac{\varepsilon}{2} \le \varepsilon, 
    \end{align*}
which gives the cyclicity of $f\varphi$ for $S$ in $\mathcal{X}$.
\end{proof}

\begin{coro}\label{puissCycl}
    Let $f \in \A$ and assume that $f$ is cyclic for $S$ in $\X$. Then, for every $N \in \N^*$, $f^N$ is cyclic for $S$ in $\X$. 
\end{coro}

\begin{proof}
     We proceed by induction. Suppose that $f^N$ is cyclic for $S$ in $\X$ for some $N \in \N^*$. Then, from Lemma \ref{cyclProdXA}, we get that $f^{N+1}$ is cyclic for $S$ in $\X$.
\end{proof}
%
%
%
\subsection{Some operator tools}\label{subsection-operator-tools}
For Theorem~\ref{theorem-2}, we introduce now two more assumptions that we shall need in order to use a tauberian result of Atzmon. Apart assumptions (H1),(H2),...,(H6), we also assume that: 
\begin{enumerate}
\item[(H7)] The function $\chi_1\in\A$.
\item[(H8)] There exists $C>0$ and $p\in\mathbb N$ such that for every $n\geq 0$, $\|\chi_n\|_\A\leq C n^p$.
\end{enumerate}
In the following, $\sigma(\chi_1)$ denotes the spectrum of $\chi_1$ in the Banach algebra $\A$.
\begin{lem}\label{lem:spectreS}
We have
\begin{enumerate}
\item[(i)] $\sigma(\chi_1)=\overline{\D}$. 
\item[(ii)] $\Hol(\overline{\D})\subset\A$.
\item[(iii)] For every $|\lambda|>1$, we have
\[
\|(z-\lambda)^{-1}\|_{\A}\lesssim \frac{|\lambda|^{p+1}}{(|\lambda|-1)^{p+1}}.
\]
\end{enumerate}
\end{lem}
\begin{proof}
$(i)$: According to (H8), we have 
\begin{equation}\label{eq:consequence-immediate-H8-shift}
\|\chi_1^n\|_\A=\|\chi_n\|_\A\leq C n^p.
\end{equation}
It immediately implies that the spectral radius of $\chi_1$ satisfies $r(\chi_1)\leq 1$, and then $\sigma(\chi_1)\subset\overline{\D}$. On the other hand, let $\lambda\in\D$ and assume that $\lambda-\chi_1$ is invertible in $\A$. It means that there exists $f\in\A$ such that $(\lambda-\chi_1)f=1$. Evaluating at $\lambda$ gives a contradiction. Hence $\D\subset\sigma(\chi_1)$ and we conclude by closeness of $\sigma(\chi_1)$. 

$(ii)$: Let $f\in \Hol(\overline{\D})$ and write 
\[
f(z)=\sum_{n=0}^\infty a_n z^n,\qquad z\in\overline{\D}.
\]
It is well known that there exists a constant $c>0$ such that  $a_n=O(\exp(-cn))$ as $n\to \infty$. See for instance \cite[Theorem 5.7]{DBR1}. Thus according to (H8), we have
\[
|a_n|\|\chi_n\|_\A\lesssim n^p \exp(-cn).
\]
Thus $g=\displaystyle\sum_{n=0}^\infty a_n\chi_n$ defines a function in $\A$. Now using that convergence in $\A$ implies pointwise convergence, we see that $f=g$, whence $f\in\A$. 

$(iii)$: Observe that for $|\lambda|>1$, the function $z\longmapsto \frac{1}{z-\lambda}$ is in $\Hol(\overline{\D})$, and we have 
\[
\frac{1}{z-\lambda}=-\sum_{n=0}^\infty \frac{z^n}{\lambda^{n+1}}.
\]
Thus, using (H8), 
\[
\begin{aligned}
\|(z-\lambda)^{-1}\|_{\A} &\leq \sum_{n=0}^\infty \frac{\|\chi_n\|_\A}{|\lambda|^{n+1}}\\
&\leq C \sum_{n=0}^\infty \frac{n^p}{|\lambda|^n}.
\end{aligned}
\]
It is easy to check that for every $0<x<1$ and every $p\in\mathbb N$, we have $\displaystyle\sum_{n=0}^\infty n^px^n\leq \frac{p!}{(1-x)^{p+1}}$, whence
\[
\|(z-\lambda)^{-1}\|_{\A} \lesssim \frac{p!}{\left(1-\frac{1}{|\lambda|}\right)^{p+1}}=\frac{p!|\lambda|^{p+1}}{(|\lambda|-1)^{p+1}},
\]
which concludes the proof.
\end{proof}

\begin{rem}
Using similar arguments as in the proof of Lemma~\ref{lem:spectreS} $(i)$, we can prove that in our context when $\X$ and $\A$ satisfy (H1) to (H8), then $\sigma(S)=\overline{\D}$. 
\end{rem}

We now recall the following famous result of Atzmon \cite{Atzmon}.
\begin{theo}[Atzmon]\label{ThmAtzmon}
    Let $E$ be a Banach space and let  $T \in \Lcont(E)$ whose spectrum $\sigma(T) = \{ \zeta_0\}$ for some $\zeta_0\in\T$. Suppose that there exists $k \ge 0$ such that 
    \begin{equation}\label{eq:Atzmon-hypothesis}
    \norm{T^n} \underset{n \rightarrow + \infty}{=} O(n^k) \text{~and~} \log \norm{T^{-n}}  \underset{n \rightarrow + \infty}{=} o(\sqrt{n}).
    \end{equation} 
    Then $(T - \zeta_0 I)^k = 0$.
\end{theo}

Using Cauchy's inequalities, we can obtain the following version which translates growth assumptions \eqref{eq:Atzmon-hypothesis} on the iterates of $T$ into growth assumptions on the resolvant of $T$. It appears in \cite{TheseVectCycl} and we refer to it for the proof of this corollary.

\begin{coro}\label{coroThmAtzmon}
 Let $E$ be a Banach space and let  $T \in \Lcont(E)$ whose spectrum $\sigma(T) = \{ \zeta_0\}$ for some $\zeta_0\in\T$. Suppose that there exists $k \ge 0$ and $c > 0$ such that
 \[
 \norm{(T-\lambda I)^{-1}} \le \frac{c \abs{\lambda}^k}{(\abs{\lambda} - 1)^k}, \quad \abs{\lambda} > 1,
 \]
 and for all $\varepsilon > 0$, there exists $K_\varepsilon > 0$ such that 
 \[
 \norm{(T-\lambda I)^{-1}} \le K_\varepsilon \exp\left(\frac{\varepsilon}{1-\abs{\lambda}}\right), \quad \abs{\lambda} < 1.
 \] 
Then $(T-\zeta_0I)^k = 0$.
\end{coro}

\section{Proof of Theorem~\ref{theorem-1} and Corollary~\ref{corollaire-1}}\label{cyclicity}

We restate Theorem~\ref{theorem-1} in the following.
\begin{theo}\label{theorem-1bis}
Let $\X$ satisfying (H1) to (H3) and let $\mathcal A$ satisfying  (H4) to  (H6). Then there exists $N\in\mathbb N^*$ such that for every $f,g\in\mathcal A$ satisfying  $|g(z)|\leq |f(z)|$ for every $z\in\D$, we have $[g^N]_\X\subset [f]_\X$.
\end{theo}
\begin{proof}
The proof follows the same lines as  \cite[Theorem 2.5]{Egueh-Kellay-Zarrabi}. For $\lambda \in \C^*$ define 
\[
\delta_\lambda := \inf_{z \in \D} \left \{ \abs{1 - \lambda g(z)} + \abs{f(z)} \right \}.
\] 
Let $z \in \D$. If $\abs{g(z)} \le \frac{1}{2\abs{\lambda}}$, then 
\[
\abs{1 - \lambda g(z)} \ge 1 - \abs{\lambda}\abs{g(z)} \ge 1 - \frac{1}{2} = \frac{1}{2}.
\]
Otherwise if $\abs{g(z)} \ge \frac{1}{2\abs{\lambda}}$, then 
\[
\abs{f(z)} \ge \abs{g(z)} \ge \frac{1}{2 \abs{\lambda}}.
\]
In particular, we get
\[
\delta_\lambda \ge \min \left(\frac{1}{2}, \frac{1}{2\abs{\lambda}} \right) > 0.
\]
 Let us now define $M_\lambda = 1 + \abs{\lambda} \norm{g}_\infty + \norm{f}_\infty$. Note that according to (H4) and Lemma \ref{multpXHinftyX}, $f,g\in H^\infty$, and then $M_\lambda$ is finite. Thus we have 
\[
0<\frac{\delta_\lambda}{M_\lambda} \le \frac{\abs{1 - \lambda g(z)} + \abs{f(z)}}{M_\lambda} \le \frac{1 + \abs{\lambda}\norm{g}_\infty + \norm{f}_\infty}{M_\lambda} = 1.
\] 
From (H6), there exists two functions $\tilde{G}_\lambda, \tilde{F}_\lambda \in \A$ such that $$(1 - \lambda g)\frac{\tilde{G}_\lambda}{M_\lambda} + f\frac{\tilde{F}_\lambda}{M_\lambda} \equiv 1 \text{~on~} \D \text{~and~} \|\tilde{G}_\lambda\|_\A, \|\tilde{F}_\lambda\|_\A \le \frac{CM_\lambda^A}{\delta_\lambda^A}.$$
    Take now $G_\lambda = \frac{\tilde{G}_\lambda}{M_\lambda} \in \A$ and $F_\lambda = \frac{\tilde{F}_\lambda}{M_\lambda} \in \A$. We thus obtain  
    \begin{equation}\label{coronaThm1}
        (1 - \lambda g)G_\lambda + fF_\lambda \equiv 1 \text{~on~} \D \text{~and~} \norm{G_\lambda}_\A, \norm{F_\lambda}_\A \le \frac{CM_\lambda^{A-1}}{\delta_\lambda^A}.
    \end{equation}
    
 We may assume that the closed ideal $\mathcal I_f$ of $\A$ defined by \eqref{def-ideal} is a proper ideal of $\A$, otherwise by Lemma~\ref{IdealfIncluCyclf}, the result is trivial.  Let us introduce the quotient map $\Pi : \A \longrightarrow \A / \I_f$ where we recall that the quotient space $\A / \I_f$ is endowed with the usual norm 
 \[
 \norm{\Pi(h)}_{\A / \I_f} := \text{dist}(h, \I_f) = \inf_{\varphi \in \I_f} \norm{h-\varphi}_\A,\qquad h\in\A,
 \] 
 making $\A/\mathcal I_f$ a Banach algebra. Using \eqref{coronaThm1} and the fact that $fF_\lambda\in\mathcal I_f$, we see that $\Pi(1-\lambda g)\Pi(G_\lambda)=\mathbbm{1}$ for every $\lambda\in\C^*$, where $\mathbbm{1}$ is the unit of $\A/\mathcal I_f$. Thus $\Pi(1-\lambda g)$ is invertible for every $\lambda\in\C$, and for $\lambda\in\C^*$, we have $\Pi(1-\lambda g)^{-1}=\Pi(G_\lambda)$. In particular, with \eqref{coronaThm1}, we get that,  for every $\lambda\in\C^*$, 
 \[
 \norm{\Pi(1-\lambda g)^{-1}}_{\A / \I_f} = \norm{\Pi(G_\lambda)}_{\A / \I_f} \le \norm{G_\lambda}_\A \lesssim \frac{M_\lambda^{A-1}}{\delta_\lambda^A}.
 \]
 Now let $\ell\in \left (\A / \I_f \right)^*$, $\ell\neq 0$, and define 
 \[
 \varphi(\lambda)= \scal{\Pi(1-\lambda g)^{-1}}{\ell},\qquad \lambda\in\C.
 \] 
 The function $\varphi$ is an entire function. Moreover, for $|\lambda|>1$, $\delta_\lambda=\frac{1}{2|\lambda|}$, and we obtain
 \[
 |\varphi(\lambda)|\lesssim  \norm{\Pi(1-\lambda g)^{-1}}_{\A / \I_f}\lesssim M_\lambda^{A-1}|\lambda|^A.
 \]
 Since $M_\lambda=1+\|f\|_\infty+|\lambda|\|g\|_\infty$, we see that 
 \begin{equation}\label{eq:croissance-varphi}
 |\varphi(\lambda)|=O(|\lambda|^{2A-1}), \quad |\lambda|\to\infty.
 \end{equation}
 We may of course assume that $g\neq 0$. For $|\lambda|<\|g\|^{-1}_{\A}$, note that 
 \[
 (1-\lambda g)^{-1}=\sum_{n=0}^\infty \lambda^n g^n,
 \]
 and thus
\[
\varphi(\lambda)=\sum_{n=0}^\infty  \langle \Pi(g^n)|\ell\rangle\lambda^n.
\]
It follows from \eqref{eq:croissance-varphi} and Liouville's theorem that, for every $n>2A-1$, we have $\langle \Pi(g^n)|\ell\rangle=0$. In particular, for $N=[2A]$, we have $\langle \Pi(g^N)|\ell\rangle=0$. Since this holds for any $\ell\in \left (\A / \I_f \right)^*$, we get that $\Pi(g^N)=0$, that is $g^N\in\I_f$. It then follows from Lemma~\ref{IdealfIncluCyclf} that $[g^N]_\X\subset [f]_\X$, which concludes the proof.
 \end{proof}

\begin{rem}
Note that we can take $N=[2A]$ in Theorem~\ref{theorem-1bis}, where $A\geq 1$ is the constant given in (H6).
\end{rem}

We restate Corollary~\ref{corollaire-1} in the following. 
\begin{coro}\label{corollaire-1-bis}
Let $\X$ satisfying  (H1) to  (H3) and let $\mathcal A$ satisfying (H4) to (H6). Let $f,g\in\mathcal A$ such that for every $z\in\D$, we have $|g(z)|\leq |f(z)|$. Suppose $g$ is cyclic for $S$ in $\X$. Then $f$ is cyclic for $S$ in $\X$. 
\end{coro}

\begin{proof}
    Since $g \in \A$ is cyclic for $S$ in $\X$, we get from Corollary \ref{puissCycl} that $g^m$ is also cyclic for $S$ in $\X$ for every $m \in \N^*$. Moreover, according to Theorem \ref{theorem-1bis}, there exists $N \in \N^*$ such that $[g^N]_\X \subset [f]_\X$. Then the cyclicity of $g^N$ immediately implies the cyclicity of $f$. 
\end{proof}

\begin{coro}
Let $\X$ satisfying  (H1) to  (H3) and let $\mathcal A$ satisfying (H4) to (H6). Let $f\in\mathcal A$ and assume that $\inf_{\D}|f(z)|>0$. Then $f$ is cyclic for $S$ in $\X$.
\end{coro}
\begin{proof}
Denote by $\delta=\inf_{\D}|f(z)|>0$. Since $\mathcal P$ is dense in $\X$, the constant function $g=\delta\chi_0$ is cyclic  for $S$ in $\X$. The conclusion now follows from Corollary \ref{corollaire-1-bis}. 
\end{proof}

We can weaken the assumption $|g(z)|\leq |f(z)|$ in Theorem~\ref{theorem-1bis} and obtain a similar conclusion.
\begin{theo}\label{theo-evgue-logarithm-growth}
Let $\X$ satisfying (H1) to (H3) and let $\mathcal A$ satisfying  (H4) to  (H6). Let $f,g\in\mathcal A$ and assume that $\Re(g)\geq 0$ and there exists $\gamma>1$ such that 
\[
|g(z)|\leq \left(\log\frac{\|f\|_\A}{|f(z)|}\right)^{-\gamma},\qquad z\in\D.
\]
Then we have $[g]_\X\subset [f]_\X$.
\end{theo}
\begin{proof}
The proof is similar to the proof of Theorem~\ref{theorem-1bis}, but we replace Liouville's theorem by a Phragmen--Lindel\"of theorem. The details are left to the reader.
\end{proof}

\begin{rem}
Theorem \ref{theo-evgue-logarithm-growth} appears in  \cite{Egueh-Kellay-Zarrabi} in the case when $\X=\mathcal D_\alpha^p$ the Besov--Dirichlet space and $\A=\mathcal D_\alpha^p\cap A(\D)$. The proof is similar in our general context.
\end{rem}

\section{Proof of Theorem~\ref{theorem-2}}\label{Section4}

We restate Theorem~\ref{theorem-2} in the following. 
\begin{theo}\label{theorem-2-bis}
Let $\X$ satisfying (H1) to  (H3) and let $\mathcal A$ satisfying (H4) to (H8). Assume that there exists $\zeta_0\in\T$ such that $z-\zeta_0$ is cyclic for $S$ in $\X$. Let $f\in\mathcal A\cap A(\D)$ be an outer function such that $\mathcal Z(f)=\{\zeta_0\}$. Then $f$ is cyclic for $S$ in $\X$. 
\end{theo}

\begin{proof}
The proof follows the same lines as  \cite[Theorem 1.1]{Egueh-Kellay-Zarrabi}. Since $f$ is outer, we have $(1 - \abs{z})\log\abs{f(z)}\to 0$ as $|z|\to 1$. See \cite{ShapiroShields}.
    In other words, for all $\varepsilon > 0$, there exists $c_\varepsilon > 0$ such that for every $z \in \D$, 
 \begin{equation}\label{eq:controle-outer-function}
 \abs{f(z)} \ge c_\varepsilon \exp\left(-\frac{\varepsilon}{2(1 - \abs{z})}\right).
\end{equation}
For $\lambda\in\C$, $|\lambda|\neq 1$, define
$$\delta_\lambda := \inf_{z \in \D} \left \{ \abs{\lambda - z} + \abs{f(z)} \right \}.$$
 If $\abs{z - \lambda} \le \frac{\abs{1 - \abs{\lambda}}}{2}$, then 
\[
\frac{\abs{1 - \abs{\lambda}}}{2} \ge \abs{\abs{z} - \abs{\lambda}} \ge \abs{1 - \abs{\lambda}} -\abs{1 - \abs{z}}
\]
whence $1 - \abs{z} \ge \frac{\abs{1 - \abs{\lambda}}}{2}$. We thus get from \eqref{eq:controle-outer-function} that, if $\abs{z - \lambda} \le \frac{\abs{1 - \abs{\lambda}}}{2}$, then 
\[
\abs{f(z)} \ge  c_\varepsilon \exp\left(-\frac{\varepsilon}{2(1 - \abs{z})}\right) \ge  c_\varepsilon \exp\left(-\frac{\varepsilon}{1 - \abs{\lambda}}\right).
\]
In particular, we deduce that for $|\lambda|\neq 1$, we have
    \begin{equation}\label{deltaMin}
        \delta_\lambda \ge \min \left (c_\varepsilon \exp\left(-\frac{\varepsilon}{1 - \abs{\lambda}}\right), \frac{\abs{1 - \abs{\lambda}}}{2} \right) > 0.
    \end{equation}
    Let us now define  $M_\lambda= 1 + \abs{\lambda} + \norm{f}_\infty$. Then there exists $F_\lambda,G_\lambda\in\A$ such that
    \begin{equation}\label{courThm2}
        (z-\lambda)G_\lambda + f F_\lambda \equiv 1 \text{~on~} \D \text{~and~} \norm{G_\lambda}_\A, \norm{F_\lambda}_\A \le C \frac{M_\lambda^{A-1}}{\delta_\lambda^A}.
    \end{equation} 
We may assume that $\I_f\neq \A$, otherwise, according to Lemma~\ref{IdealfIncluCyclf} and (H7), $\mathcal P\subset [f]_\X$, and we get the cyclicity of $f$. Let us introduce the quotient map $\Pi : \A \longrightarrow \A / \I_f$ and the operator 
\[
\begin{array}{rccl}
T:&\A/\I_f&\longmapsto &\A/\I_f \\
&\Pi(g)&\longrightarrow & \Pi(z)\Pi(g).
\end{array}
\]
We want to apply Corollary \ref{coroThmAtzmon} to $T$. For this purpose, we need to check that $T$ satisfies the following three conditions:
\begin{enumerate}
\item[(i)] Its spectrum $\sigma(T)=\{\zeta_0\}$. 
\item[(ii)] There exists $k \ge 0$ and $c > 0$ such that
 \[
 \norm{(T-\lambda I)^{-1}} \le \frac{c \abs{\lambda}^k}{(\abs{\lambda} - 1)^k}, \quad \abs{\lambda} > 1,
 \]
 \item[(iii)] For all $\varepsilon > 0$, there exists $K_\varepsilon > 0$ such that 
 
 \[
 \norm{(T-\lambda I)^{-1}} \le K_\varepsilon \exp\left(\frac{\varepsilon}{1-\abs{\lambda}}\right), \quad \abs{\lambda} < 1.
 \] 
\end{enumerate}
(i): Since $T$ is the multiplication operator by $\Pi(z)$ on the Banach algebra $\A/\I_f$, it is easy to check that $\sigma(T)\subset \sigma(\Pi(z))$. Now, since $\sigma(T)\neq\emptyset$, it remains to show that $\sigma(\Pi(z))\subset \{\zeta_0\}$. So take $\mu \in \C \backslash \{\zeta_0\}$. Since $z \longmapsto \abs{z - \mu} + \abs{f(z)}$ is continuous on the compact set $\overline{\D}$, there exists $z_0 \in \overline{\D}$ such that
\[
\delta=\inf_{z \in \D} \{ \abs{z - \mu} + \abs{f(z)}\}= \abs{z_0 - \mu} + \abs{f(z_0)}.
\] 
Using that $\mathcal Z(f)=\{\zeta_0\}$ and $\mu\neq \zeta_0$, we see that $\delta>0$. Therefore, by (H6), there exists $g_1,g_2\in\A$ such that 
$(z-\mu)g_1 + fg_2 \equiv 1$ on $\D$. In particular, $(\Pi(z) - \mu \mathbbm{1})\Pi(g_1) = \mathbbm{1}$ and $\Pi(z) - \mu \mathbbm{1}$ is invertible in $\A/\I_f$. Hence $\mu \in \C \backslash \sigma(\Pi(z))$ and $\C \backslash \{\zeta_0\} \subset \C \backslash \sigma(\Pi(z))$. In other words,  $\sigma(\Pi(z)) \subset \{ \zeta_0 \}$, and thus $\sigma(T)=\{\zeta_0\}$. \\

\noindent
(ii): Let $\abs{\lambda} > 1$. According to  Lemma~\ref{lem:spectreS}, $(z-\lambda)^{-1}\in\A$. 
Since $(T-\lambda I)^{-1}$ is the multiplication operator by $\pi((z-\lambda)^{-1})$ on $\A/\I_f$, we have 
 \[
 \norm{(T - \lambda I)^{-1}} \leq \norm{\Pi((z- \lambda)^{-1})}_{\A/\I_f} \leq \|(z-\lambda)^{-1}\|_\A.
 \]
Lemma~\ref{lem:spectreS} implies now that 
\[
\norm{(T - \lambda I)^{-1}}\lesssim \frac{\abs{\lambda}^{p+1}}{(\abs{\lambda} - 1)^{p+1}}.
\]

 \noindent
(iii): Let $\abs{\lambda} < 1$. Since $(T-\lambda I)^{-1}$ is the multiplication operator by $(\Pi(z)-\lambda \mathbbm{1})^{-1}$ on $\A/\I_f$, we have 
 \[
 \norm{(T - \lambda I)^{-1}}\leq \norm{(\Pi(z) - \lambda \mathbbm{1})^{-1}}_{\A/\I_f}.
 \]
 But, from \eqref{courThm2}, we have $\Pi(z-\lambda)\Pi(G_\lambda)=\mathbbm{1}$. Thus $(\Pi(z) - \lambda \mathbbm{1})^{-1} = \Pi(G_\lambda)$, which implies 
 \[
 \norm{(T-\lambda I)^{-1}} \le \norm{\Pi(G_\lambda)}_{\A / \I_f} \le \norm{G_\lambda}_\A \le C \frac{M_\lambda^{A-1}}{\delta_\lambda^A}.
 \]
Observe that, for $\abs{\lambda} < 1$, we have $1 \le M_\lambda = 1 + \abs{\lambda} + \norm{f}_\infty < 2 + \norm{f}_\infty$, which gives that 
 \[
 \|(T-\lambda I)^{-1}\|\lesssim \frac{1}{\delta_\lambda^A}.
 \]
Let us remark that 
\[
\frac{2c_\varepsilon}{1 - \abs{\lambda}}\exp\left(-\frac{\varepsilon}{1 - \abs{\lambda}}\right)\to 0,\quad\mbox{as }|\lambda|\to 1^{-}.
\]
Then, for all $\varepsilon > 0$, there exists $K_\varepsilon' > 0$ such that for all $\lambda \in \D$, 
\[
\frac{2c_\varepsilon}{1 - \abs{\lambda}}\exp\left(-\frac{\varepsilon}{1 - \abs{\lambda}}\right) \le K_\varepsilon',
\]
that is
\[
\frac{c_\varepsilon}{K_\varepsilon'}\exp\left(-\frac{\varepsilon}{1 - \abs{\lambda}}\right)  \le \frac{1 - \abs{\lambda}}{2}.
\]

 Therefore, from \eqref{deltaMin}, we get that 
 \[
 \delta_\lambda \ge K_\varepsilon'' \exp\left(-\frac{\varepsilon}{1 - \abs{\lambda}}\right)
 \] 
 where $K_\varepsilon'' = \min \left (c_\varepsilon, \frac{c_\varepsilon}{K_\varepsilon'} \right)>0$. Finally, we obtain 
 \[
 \norm{(T-\lambda I)^{-1}} \lesssim \frac{1}{K_\varepsilon^{''A}} \exp\left(\frac{A\varepsilon}{1 - \abs{\lambda}}\right).
 \]
Changing $\varepsilon$ by $\varepsilon/A$ if necessary, we then deduce (iii).\\

Therefore, the operator $T$ satisfies the assumptions of Corollary \ref{coroThmAtzmon}, and we get that $(T - \zeta_0 I)^{p+1} = 0$. In other words, 
\[
(\Pi(z) - \zeta_0 \mathbbm{1})^{p+1} = \Pi((z - \zeta_0)^{p+1})= 0.
\]  
Thus $(z - \zeta_0)^{p+1} \in \I_f \subset [f]_\X$. Since $z - \zeta_0$ is cyclic for $S$ in $\X$, we get from Corollary \ref{puissCycl} that $(z - \zeta_0)^{p+1}$ is also cyclic for $S$ in $\X$. Finally, $[f]_\X = \X$ and $f$ is cyclic for $S$ in $\X$.  
\end{proof}

The assumption that $z-\zeta_0$ is cyclic for $S$ in $\X$ (in Theorem~\ref{theorem-2}) is linked with the point spectrum $\sigma_p(S^*)$ of $S^*$ where $S^*$ is the adjoint operator of $S:\X\longrightarrow\X$, and with the property known as \emph{bounded point evaluation}. We say that $\zeta\in\T$ is a bounded point evaluation of $\X$ if there exists a constant $C>0$ such that for every $p\in\mathcal P$, we have 
\[
|p(\zeta)|\leq C\|p\|_\X.
\]
Since the polynomials are dense in $\X$, this means that the functional $p\longmapsto p(\zeta)$ extends uniquely to a continuous functional on $\X$. 
\begin{lem}\label{lem:point-spectrum-S-b-cyclicity}
Let $\X$ satisfying  (H1) to  (H3) and let $\zeta\in\T$. The following assertions are equivalent:
\begin{enumerate}
\item[(i)] The point $\bar\zeta\in \sigma_p(S^*)$. 
\item[(ii)] $\zeta$ is a bounded point evaluation of $\X$.
\item[(iii)] The function $z-\zeta$ is not cyclic for $S$ in $\X$.

\end{enumerate}
\end{lem} 
\begin{proof}
$(i)\implies (ii)$: Let $\bar\zeta\in\sigma_p(S^*)$. Hence there is $k_\zeta\in\X^*$, $k_\zeta\neq 0$, such that $S^* k_\zeta=\bar\zeta k_\zeta$. In particular, on one hand, we have, for every $k\geq 0$, 
\[
\langle z^k|S^*k_\zeta\rangle=\langle z^{k+1}|k_\zeta \rangle,
\]
and on the other hand, we also have 
\[
\langle z^k| \bar\zeta k_\zeta \rangle=\zeta \langle z^k | k_\zeta \rangle.
\]
Hence for every $k\geq 0$, 
\[
\langle z^{k+1}| k_\zeta \rangle=\zeta \langle z^k| k_\zeta \rangle,
\]
and by induction, we get
\[
\langle z^k|k_\zeta\rangle=\zeta^k \langle 1|k_\zeta \rangle.
\]
Observe that necessarily $\langle 1|k_\zeta \rangle\neq 0$, otherwise the previous relation would imply that $k_\zeta$ vanishes on the set of polynomials which is dense in $\X$, but that contradicts the fact that $k_\zeta\neq 0$. Hence normalizing $k_\zeta$ if necessary, we may assume that $\langle 1|k_\zeta \rangle=1$ and we thus deduce 
\[
\langle z^k|k_\zeta\rangle=\zeta^k,\qquad k\geq 0.
\]
By linearity, for every $p\in\mathcal P$, we have 
\[
\langle p |k_\zeta \rangle=p(\zeta),
\]
and finally we obtain
\[
|p(\zeta)| \leq \| k_\zeta\|_{\X^*} \|p\|_\X.
\]

$(ii)\implies (iii)$: Assume that there exists a constant $C>0$ such that for every $p\in\mathcal P$, we have 
\begin{equation}\label{eq:3343DSDQ2E3E4}
|p(\zeta)|\leq C \|p\|_\X,
\end{equation}
and argue by absurd, assuming also that $z-\zeta$ is cyclic for $S$. Let $\varepsilon>0$. Then there exists a polynomial $q$ such that 
\[
\|(z-\zeta)q-1\|_\X\leq\varepsilon.
\]
Consider the polynomial $p=(z-\zeta)q-1$ and observe that $p(\zeta)=-1$. According to \eqref{eq:3343DSDQ2E3E4}, we thus have
\[
1\leq C \varepsilon,
\]
and this gives a contradiction for sufficiently small $\varepsilon>0$.

$(iii)\implies (i)$: Assume that $z-\zeta$ is not cyclic for $S$ in $\X$. According to Hahn--Banach Theorem, there exists $\varphi\in\X^*$, $\varphi\neq 0$ such that $\varphi$ vanishes on $[z-\zeta]_\X$. In particular, for every $k\geq 0$, we have
\[
\langle z^k|S^*\varphi-\bar\zeta\varphi\rangle=\langle z^k(z-\zeta)|\varphi\rangle=0.
\]
By linearity, we get that for every $p\in\mathcal P$, 
\[
\langle p|S^*\varphi-\bar\zeta\varphi\rangle=0,
\]
and since $\mathcal P$ is dense in $\X$, we deduce that $S^*\varphi=\bar\zeta\varphi$. But $\varphi\neq 0$, and thus $\bar\zeta\in\sigma_p(S^*)$. 
\end{proof}

\section{Some concrete examples}\label{examples}

We study in this section two applications.

\subsection{De Branges-Rovnyak spaces}

To every non-constant function $b$ in the closed unit ball of $H^\infty$, we associate the de Branges--Rovnyak space $\mathcal H(b)$ 
defined as the reproducing kernel Hilbert space on $\D$  with positive definite kernel given by 
\[
k_\lambda^b(z) = \frac{1 - \overline{b(\lambda)}b(z)}{1 - \overline{\lambda}z},\qquad \lambda,z \in \D.
\]
It is well-known that $\mathcal H(b)$ is contractively contained into $H^2$, and moreover it is invariant with respect to $S$ if and only if $\log(1-|b|)\in L^1(\T)$ \cite[Corollary 20.20]{DBR2}.

So {\bf from now on}, we assume that $b$ is a non-constant function in the closed unit ball of $H^\infty$ which satisfies $\log(1-|b|)\in L^1(\T)$, and we denote by $S_b$ the restriction of the shift operator on $\mathcal H(b)$. Note that, for every $\lambda\in\D$, the evaluation map $f\longmapsto f(\lambda)$ is continuous on $\mathcal H(b)$. It is also known that when $\log(1-|b|)\in L^1(\T)$, the set of polynomials $\mathcal P$ is dense in $\mathcal H(b)$. Hence $\mathcal H(b)$ satisfies (H1) to (H3).
We refer the reader to \cite{DBR2,MR1289670} for an in-depth study of de Branges-Rovnyak spaces and their connections to numerous other topics in operator theory and complex analysis. 

Now consider $\A=\mathfrak{M}(\mathcal H(b))$ the Banach algebra of multipliers of $\mathcal H(b)$. Of course $\A$ satisfies (H4) and also (H5) according to Lemma~\ref{multpXHinftyX}. Now we immediately get from Theorem~\ref{theorem-1bis} and Corollary~\ref{corollaire-1-bis} the following result. 
\begin{theo}\label{Thm1-DBR}
Let $b$ be a function in the closed unit ball of $H^\infty$ such that $\log(1-|b|)\in L^1(\T)$. Assume that $\mathfrak{M}(\mathcal H(b))$ satisfies (H6). Let $f,g\in\mathfrak{M}(\mathcal H(b))$ which satisfies $|g(z)|\leq |f(z)|$ for every $z\in\D$. 
\begin{enumerate}
\item[$(1)$] There exists $N\in\mathbb N^*$ such that 
\[
[g^N]_{\mathcal H(b)}\subset [f]_{\mathcal H(b)}.
\] 
\item[$(2)$] Moreover, if $g$ is cyclic for $S_b$ in $\mathcal H(b)$, then $f$ is also cyclic for $S_b$ in $\mathcal H(b)$. 
\end{enumerate}
\end{theo}

\begin{proof}
Since $\mathcal H(b)$ satisfies (H1) to (H3) and $\mathfrak{M}(\mathcal H(b))$ satisfies (H4) to (H6), it suffices to apply Theorem~\ref{theorem-1bis} and Corollary~\ref{corollaire-1-bis}.
\end{proof}
The previous result leads to the following question.
\begin{quest}
Let $b$ be a function in the closed unit ball of $H^\infty$ such that $\log(1-|b|)\in L^1(\T)$. Can we characterize those $b$ such that $\mathfrak{M}(\mathcal H(b))$ satisfies (H6)?
\end{quest}

\begin{rem}
In \cite[Theorem 6.5]{CoronaHbMult}, the authors prove that when $b$ is a rational (not inner) function in the closed unit ball of $H^\infty$, then $\mathfrak{M}(\mathcal H(b))$ satisfies (H6) with some constant $A>2+m$ and $m$ is the maximum of the multiplicities of the zeros of the pythagorean mate $a$ of $b$ (see \eqref{eq:factorization-mate} for the definition of $a$). So our results apply in this case. It would be interesting to see if we could have more examples. It should also be noted that the cyclicity of $S_b$ in $\mathcal H(b)$ has been studied recently in \cite{CyclBG,CyclFG,CyclLebreton-Fricain} where different technics were developed. In particular, in the case when $b$ is a rational (not inner) function in the closed unit ball of $H^\infty$, the cyclic vectors have been completely characterized. However, even in this case, Part $(1)$ of Theorem~\ref{Thm1-DBR} seems to be new.
\end{rem}
With regards to the application of Theorem~\ref{theorem-2-bis}, we need to recall the notion of angular derivatives. We say that $b$ has an \textit{angular derivative in the sense of Carathéodory} at $\zeta \in \T$ if $b$ and $b'$ both have a non-tangential limit at $\zeta$ and $\abs{b(\zeta)} = 1$. We denote by $E_0(b)$ the set of such points. It is known that for $\zeta \in \T$, every function $f \in \mathcal{H}(b)$ has a non-tangential limit at $\zeta$ if and only if $\zeta \in E_0(b)$. In particular, if $\zeta\in E_0(b)$, then there exists $C>0$ such that 
\begin{equation}\label{eq:evaluation-ADC}
|f(\zeta)|\leq C \|f\|_{\mathcal H(b)},\qquad f\in\mathcal H(b).
\end{equation}
See \cite[Theorem 21.1]{DBR2}. It is also known \cite[Theorem 28.37]{DBR2} that for $\zeta\in\T$, we have 
\begin{equation}\label{eq:pt-spectrum-Sb}
\mbox{ $\bar\zeta$ is an eigenvalue for $S_b^*$ if and only if $\zeta\in E_0(b)$}. 
\end{equation}
We then get from Theorem~\ref{theorem-2-bis} the following result.
\begin{theo}
Let $b$ be a function in the closed unit ball of $H^\infty$ such that $\log(1-|b|)\in L^1(\T)$. Assume that $\mathfrak{M}(\mathcal H(b))$ satisfies (H6) and (H8). Let $f\in\mathfrak{M}(\mathcal H(b))\cap A(\D)$ and assume that $\mathcal Z(f)=\{\zeta_0\}$ for some $\zeta_0\in\T$. Then the following assertions are equivalent.
\begin{enumerate}
\item[$(i)$] The function $f$ is cyclic for $S_b$.
\item[$(ii)$] The function $f$ is outer and $\zeta_0\notin E_0(b)$. 
\end{enumerate}
\end{theo}
\begin{proof}
$(i)\implies (ii)$: If $f$ is cyclic for $S_b$ in $\mathcal H(b)$, then since $\mathcal H(b)$ is contractively contained in $H^2$, the function $f$ is also cyclic for $S$ in $H^2$. Thus, by Beurling's theorem, $f$ should be outer. On the other hand, assume that $\zeta_0\in E_0(b)$. Let $\varepsilon>0$. There exists $p\in\mathcal P$ such that 
\[
\|pf-1\|_{\mathcal H(b)}\leq \varepsilon.
\]
Hence, according to \eqref{eq:evaluation-ADC}, we get
\[
|p(\zeta_0)f(\zeta_0)-1|\leq C\varepsilon.
\]
Since $f(\zeta_0)=0$, this gives $1\leq C\varepsilon$ and thus a contradiction for sufficiently small $\varepsilon$. Therefore $\zeta_0\notin E_0(b)$.\\

$(ii)\implies (i)$: Assume that $f$ is outer and $\zeta_0\notin E_0(b)$. According to \eqref{eq:pt-spectrum-Sb}, we know that $\bar\zeta_0$ is not in the point spectrum of $S_b^*$. Then it follows from Lemma~\ref{lem:point-spectrum-S-b-cyclicity} that $z-\zeta_0$ is cyclic for $S_b$. Since $\mathcal H(b)$ satisfies (H1) to (H3) and $\mathfrak{M}(\mathcal H(b))$ satisfies (H4) to (H8), we can apply Theorem~\ref{theorem-2-bis} to get that $f$ is cyclic for $S_b$ in $\mathcal H(b)$. 
\end{proof}
The previous result leads to the following question.
\begin{quest}
Let $b$ be a function in the closed unit ball of $H^\infty$ such that $\log(1-|b|)\in L^1(\T)$. Can we characterize those $b$ such that  there exists $C>0$ and $p\in\mathbb N$ with 
\[
\|\chi_n\|_{\mathfrak{M}(\mathcal H(b))}\leq C n^p,\qquad \mbox{for every }n\geq 0\,?
\]
\end{quest}
We can give a positive answer in the case when $b$ is a rational (not inner) function in the closed unit ball of $H^\infty$. 
\begin{prop}
Let $b$ be a rational (not inner) function in the closed unit ball of $H^\infty$. Then there exists $C>0$ and $p\in\mathbb N$ such that 
\[
\|\chi_n\|_{\mathfrak{M}(\mathcal H(b))}\leq C n^p,\qquad \mbox{for every }n\geq 0.
\]
\end{prop}
\begin{proof}
Since $b$ is a rational (not inner) function in the closed unit ball of $H^\infty$, then it is known that $\mathfrak{M}(\mathcal H(b))=H^\infty\cap\mathcal H(b)$. See \cite{MultRangeSpace}. According to Lemma~\ref{multpXHinftyX}, we get that 
\[
\|\chi_n\|_{\mathfrak{M}(\mathcal H(b))}\lesssim \|\chi_n\|_\infty+\|\chi_n\|_{\mathcal H(b)}=1+\|\chi_n\|_{\mathcal H(b)}.
\]
Thus it is sufficient to prove that 
\begin{equation}\label{eq:98S89D9SSQDSD}
\|\chi_n\|_{\mathcal H(b)}\lesssim n^p,
\end{equation}
for some $p\in\mathbb N$. When $b$ is a rational (not inner) function in the closed unit ball of $H^\infty$, we know that there exists a unique rational outer function $a$ such that $a(0)>0$ and 
\[
|a|^2+|b|^2=1\qquad \mbox{on $\T$.}
\]
See \cite{MR3503356}. We may assume that $\|b\|_\infty=1$, otherwise $\mathcal H(b)=H^2$ (with equivalent norms) and the result is trivial. Thus $a$ has at least one zero on $\T$. Factorize $a$ as 
\begin{equation}\label{eq:factorization-mate}
a(z)=a_1(z)\prod_{i=1}^s (z-\zeta_i)^{m_i},
\end{equation}
where $\zeta_i\in\T$, $m_i\geq 1$, $s\geq 1$ and $a_1$ is rational function without zeros (and poles) in $\overline{\D}$. If the series expansion of $b/a\in \Hol(\D)$ has the form
\[
\frac{b(z)}{a(z)}=\sum_{j=0}^\infty c_j z^j,\qquad |z|<1,
\]
then it is known \cite[Theorem 24.12]{DBR2} that 
\begin{equation}\label{eq:ERDFDSFDS2E3232}
\|\chi_n\|_{\mathcal H(b)}^2=1+\sum_{j=0}^n |c_j|^2.
\end{equation}
From Cauchy's inequalities, we have 
\[
\abs{c_j} = \abs{\frac{(\frac{b}{a})^{(j)}(0)}{j!}} \le \inf_{0 < r < 1} \frac{M(r)}{r^j},
\]
where 
\[
M(r) = \sup_{\abs{z} = r} \abs{\frac{b(z)}{a(z)}},\qquad 0<r<1.
\] 
Using \eqref{eq:factorization-mate}, for $|z|=r$, we have
\[
\abs{a(z)} \gtrsim \prod_{i=1}^s \abs{z-\zeta_i}^{m_i} \ge \prod_{i=1}^s (\abs{\zeta_i} - \abs{z})^{m_i} = (1-r)^N,
\]
where $N=\sum_{i=1}^s m_i$. 
We get that $M(r)\leq (1-r)^{-N}$ and thus
 \[
 \abs{c_j} \lesssim \inf_{0 < r < 1} (1-r)^{-N} r^{-j}.
 \] 
 If we introduce $\varphi(r)=(1-r)^{-N}r^{-j}$, $0<r<1$, it is not difficult to check that $\varphi$ has minimum at $r=1-\frac{N}{j+N}=\frac{j}{j+N}$, which gives that 
 \[
 |c_j|\lesssim \left(1+\frac{j}{N}\right)^N \left(1+\frac{N}{j}\right)^j\lesssim j^N\qquad \mbox{as }j\to\infty.
 \]
 Thus 
 \[
 \sum_{j=0}^n |c_j|^2\lesssim \sum_{j=0}^n j^{2N}\leq n^{2N+1}.
 \]
 Then \eqref{eq:98S89D9SSQDSD} follows from \eqref{eq:ERDFDSFDS2E3232}, which concludes the proof.
\end{proof}

\subsection{Besov-Dirichlet spaces}

For $p \ge 1$ and $\alpha > -1$, the Besov-Dirichlet space $\mathcal{D}_\alpha^p$  consists of functions $f$ holomorphic on $\D$ satisfying 
\[
\norm{f}_{\mathcal{D}_\alpha^p}^p  := \abs{f(0)}^p + (1 + \alpha)\int_\D \abs{f'(z)}^p (1 - \abs{z}^2)^\alpha ~dA(z) < \infty. 
\]
Let us recall that for $p = 2$ and $\alpha = 1$, $\mathcal{D}_\alpha^p = H^2$ the Hardy space of the unit disc, and for $p = 2$ and $\alpha = 0$, $\mathcal{D}_\alpha^p = \mathcal{D}$, the classical Dirichlet space. This example of $\mathcal D_\alpha^p$ was studied in details by Egueh--Kellay--Zarrabi in 
\cite{Egueh-Kellay-Zarrabi}, which was a source of inspiration for us. See also \cite{TheseVectCycl}.

Let us recall that if $1<p<\alpha+1$, then $H^p$ is continuously embedded in $\mathcal{D}_\alpha^p$. Hence every outer functions $f\in H^p$ is cyclic for the shift in $\mathcal{D}_\alpha^p$. See \cite[Proposition 3.1]{CyclDirichletSpace}. On the other hand, if $p > \alpha + 2$, then $\mathcal{D}_\alpha^p \subset A(\D)$ becomes a Banach algebra, and consequently the only cyclic outer functions are the invertible functions. Thus a function $f\in\mathcal D_\alpha^p$ which vanishes at least at one point in $\overline{\D}$ is not cyclic for the shift in $\mathcal{D}_\alpha^p$. See \cite{TheseVectCycl,CyclDirichletSpace}.

We will assume {\bf from now on} that $\alpha+1\leq p\leq \alpha+2$. 
\begin{lem}\label{Lemma-D-alpha-p-hypothese}
Let $p>1$ such that $\alpha+1\leq p\leq \alpha+2$ and let $\A=\mathcal D_\alpha^p\cap A(\D)$ endowed with the norm
\[
\norm{f}_\A^p := \norm{f}_{\infty}^p + \int_\D \abs{f'(z)}^p (1 - \abs{z}^2)^\alpha ~dA(z).
\]
Then $\mathcal D_\alpha^p$ satisfies (H1) to (H3) and $\A$ satisfies (H4) to (H8).
\end{lem}
\begin{proof}
It is well known that $\mathcal D_\alpha^p$ satisfies (H1) to (H3). See \cite{MR814017,BergmanSpace,Tolokonnikov,MR1112319}. Let us check that $\A$ satisfies (H4), that is 
\begin{equation}\label{eq:Dalpha-p-condition-H4ezrzer}
\mathcal D_\alpha^p\cap A(\D)\subset\mathfrak{M}(\mathcal D_\alpha^p).
\end{equation}
Let $f\in\mathcal D_\alpha^p\cap A(\D)$ and $g\in\mathcal D_\alpha^p$. It follows from Minkowski's inequality that 
        \begin{align*}
            \left (\int_\D \abs{(fg)'(z)}^p (1 - \abs{z}^2)^\alpha ~dA(z) \right)^\frac{1}{p} & \le \left (\int_\D \abs{f'(z)}^p \abs{g(z)}^p (1 - \abs{z}^2)^\alpha ~dA(z) \right)^\frac{1}{p} \\ & +  \left (\int_\D \abs{f(z)}^p \abs{g'(z)}^p (1 - \abs{z}^2)^\alpha ~dA(z) \right)^\frac{1}{p}.
        \end{align*}
        The second term is finite since $f \in H^\infty$ and $g \in \mathcal{D}_\alpha^p$. Moreover, according to the Cauchy's formula applied to $f \in A(\D)$, for every $z \in \D$, we have
        \begin{align*}
            \abs{f'(z)} = \abs {\frac{1}{2i\pi}\int_\T \frac{f(\zeta)}{(\zeta - z)^2}~d\zeta} & \le \norm{f}_\infty \frac{1}{2\pi}\int_\T \frac{1}{\abs{\zeta - z}^2} ~\abs{d\zeta} \\ & = \norm{f}_\infty \norm{k_z}_2^2 = \frac{\norm{f}_\infty}{1 - \abs{z}^2}.
        \end{align*}
        Therefore, we obtain that the first term is finite since $g \in \A_{\alpha - p}^p$. See \cite[Proposition 1.11]{BergmanSpace}. Finally, we deduce that $fg \in \mathcal{D}_\alpha^p$ and (H4) holds true. We also deduce from 
        \eqref{eq:Dalpha-p-condition-H4ezrzer} that $\A$ is an algebra. 

It is clear that $\A$ satisfies (H5) because for every $\lambda\in\D$ and every $f\in\A$, we have
\[
|f(\lambda)|\leq \|f\|_\infty\leq \|f\|_\A.
\]

The fact that $\A$ satisfies (H6), with constant $A\geq 4$, is a deep result of Tolokonnikov \cite{Tolokonnikov}. 

Clearly $\A$ satisfies (H7). So it remains to check that $\A$ satisfies (H8). We have
\[
\begin{aligned}
\|\chi_n\|^p_\A&=\|\chi_n\|_\infty^p+\int_\D |\chi'_n(z)|^p (1-|z|^2)^\alpha\,dA(z) \\
&=1+n^p\int_\D |z|^{(n-1)p}(1-|z|^2)^\alpha\,dA(z)\\
&\leq 1+n^p\int_0^{2\pi}\int_0^1 (1-r^2)^\alpha r\,dr\,\frac{d\theta}{2\pi}\\
&\leq 1+n^p \int_0^1 (1-r)^\alpha\,dr\\
&\leq 1+\frac{n^p}{\alpha+1},
\end{aligned}
\]
which gives (H8), and concludes the proof.
\end{proof}

Using Theorem~\ref{theorem-1bis} and Corollary~\ref{corollaire-1-bis}, we recover the following result due to Egueh--Kellay--Zarrabi in \cite{Egueh-Kellay-Zarrabi}.
\begin{theo}[Egueh--Kellay--Zarrabi]\label{thm:DFDFE3020GDSF}
Let $p>1$ such that $\alpha+1\leq p\leq \alpha+2$. Let $f,g\in\mathcal D_\alpha^p\cap A(\D)$ and assume that $|g(z)|\leq |f(z)|$ for every $z\in\D$. 
\begin{enumerate}
\item[$(1)$] There exists $N\in\mathbb N^*$ such that 
\[
[g^N]_{\mathcal D_\alpha^p}\subset [f]_{\mathcal D_\alpha^p}.
\] 
\item[$(2)$] Moreover, if $g$ is cyclic for $S$ in $\mathcal D_\alpha^p$, then $f$ is also cyclic for $S$ in $\mathcal D_\alpha^p$. 
\end{enumerate}
\end{theo}
\begin{proof}
According to Lemma~\ref{Lemma-D-alpha-p-hypothese}, we can apply Theorem~\ref{theorem-1bis} and Corollary~\ref{corollaire-1-bis} which immediately gives the result. 
\end{proof}
It should be noted that Theorem~\ref{thm:DFDFE3020GDSF} is an extension of a result of Brown-Shields \cite{BrownShields}.  See also \cite{MR1079693}. 

As an application of Theorem~\ref{theorem-2-bis}, we recover now the following result of Egueh--Kellay--Zarrabi in \cite{Egueh-Kellay-Zarrabi}.
\begin{theo}[Egueh--Kellay--Zarrabi]\label{thm2-egueh-kellay-zarrabi}
Let $p>1$ such that $\alpha+1\leq p\leq \alpha+2$. Let $f\in\mathcal D_\alpha^p \cap A(\D)$  be an outer function and assume that $\mathcal Z(f)=\{\zeta_0\}$ for some $\zeta_0\in\T$. Then $f$ is cyclic for $S$ in $\mathcal D_\alpha^p$. 
 \end{theo}
\begin{proof}
It is known that for every $\zeta\in\T$, $z-\zeta$ is cyclic for $S$ in $\mathcal D_\alpha^p$. 
See \cite[Proposition 4.3.8]{TheseVectCycl}. Then, according to Lemma~\ref{Lemma-D-alpha-p-hypothese}, we can apply 
Theorem~\ref{theorem-2-bis} which gives the result. 
 \end{proof}
Note that the case of the classical Dirichlet space $\mathcal D$ was discovered by Hedenmalm and Shields \cite{MR1042516} and generalized by Richter and Sundberg \cite{MR1145733}. Theorem \ref{thm2-egueh-kellay-zarrabi} was already obtained by Kellay, Lemanach and Zarrabi in \cite{CyclDirichletSpace} for $\alpha+1<p\leq \alpha+2$ using technics from  \cite{MR1042516}. Thanks to \cite[Theorem 3]{MR1042516}, as observed in \cite{Egueh-Kellay-Zarrabi}, Theorem \ref{thm2-egueh-kellay-zarrabi} remains true under the assumption that 
$\mathcal Z(f)$ is a countable set. 

\subsection{Dirichlet type spaces} 
Given a finite positive Borel measure $\mu$ on the closed unit disc $\overline{\mathbb D}$, let 
\[
U_\mu(z)=\int_\D \left(\log\left|\frac{1-\overline{w}z}{z-w}\right|^2\right)\frac{d\mu(w)}{1-|w|^2}+\int_\T \frac{1-|z|^2}{|\zeta-z|^2}\,d\mu(\zeta),\qquad z\in\D.
\]
The function $U_\mu$ is a positive superharmonic function on $\D$ and we associate to it the Dirichlet type space $\mathcal D(\mu)$ defined as the space of analytic functions $f$ on $\D$ satisfying 
\[
\int_\D |f'(z)|^2 U_\mu(z)\,dA(z)<\infty,
\]
where $dA$ stands the normalized area measure. It is known that $\mathcal D(\mu)\subset H^2$ and if for $f\in\mathcal D(\mu)$, we define 
\begin{equation}\label{eq:norme-Dmu}
\|f\|^2_{\mathcal D(\mu)}=\|f\|_2^2+\int_\D |f'(z)|^2 U_\mu(z)\,dA(z),
\end{equation}
it is known that $\mathcal D(\mu)$ is a reproducing kernel Hilbert space. These Dirichlet type spaces are important in model theory. See \cite{habilitation-aleman}. It turns out that these spaces also enter in our general framework. The key result to check that our assumptions are satisfied is the following deep result of Shuabing Luo \cite{MR4363747}. Since the result is not exactly stated like this, we shall explain how to get this following version

\begin{theo}[S. Luo]\label{thm-Luo}
Let $\mu$ be a finite positive measure on $\overline{\D}$. There exists $C>0$ such that for every $f_1,f_2\in\mathfrak M({\mathcal D(\mu)})$ satisfying 
\[
0 < \delta \le \abs{f_1} + \abs{f_2} \le 1\qquad \text{~on~} \D,
\]
there exists $g_1,g_2\in \mathfrak M({\mathcal D(\mu)})$ such that $f_1g_1+f_2g_2\equiv 1$ on $\D$  and 
\[
\|g_1\|_{\mathfrak M(\mathcal D(\mu))},\,\|g_2\|_{\mathfrak M(\mathcal D(\mu))}\leq \frac{C}{\delta^4}.
\]
\end{theo}

\begin{proof}
Let $f_1,f_2\in \mathfrak M({\mathcal D(\mu)})$ satisfying 
\[
0 < \delta \le \abs{f_1} + \abs{f_2} \le 1\qquad \text{~on~} \D.
\]
Then it is proved in \cite{MR4363747} that, for every $h\in\mathcal D(\mu)$, there exists $\varphi_1,\varphi_2\in \mathcal D(\mu)$ such that 
\[
f_1\varphi_1+f_2\varphi_2=h,
\]
and for $\ell=1,2$, we have
\[
\|\varphi_\ell\|_{\mathcal D(\mu)}\lesssim \delta^{-4}\|h\|_{\mathcal D(\mu)}.
\]
But is known that $\mathcal D(\mu)$ is a reproducing kernel Hilbert space with a complete Nevanlinna--Pick kernel \cite{MR1911187}. Hence it satisfies the Toeplitz Corona Theorem. See \cite{MR1846055,Wick}.  Thus it follows that there  exists $g_1,g_2\in \mathfrak M(\mathcal D(\mu))$ such that 
\[
f_1g_1+f_2g_2=1,
\]
and for $\ell=1,2$, we have
\[
\|g_\ell\|_{\mathfrak M(\mathcal D(\mu))}\lesssim \delta^{-4},
\]
which concludes the proof.
\end{proof}

\begin{lem}\label{Lemma-D-mu-hypothese}
Let $\mu$ be a finite positive measure on $\overline{\D}$, let $\X=\mathcal D(\mu)$ and $\A=\mathfrak M({\mathcal D(\mu)})$. Then $\X$ satisfies (H1) to (H3) and $\A$ satisfies (H4) to (H8).
\end{lem}
\begin{proof}
It is well known that $\mathcal D(\mu)$ satisfies (H1) to (H3). See \cite{habilitation-aleman}. Moreover, $\A$ satisfies trivially (H4) and (H7), as well as (H5) according to Lemma 2.1. The hypothesis (H6) follows from Luo's theorem with $A\geq 4$. It thus remains to check that $\mathfrak M({\mathcal D(\mu)})$ satisfies (H8). According to \cite[Theorem IV.1.9]{habilitation-aleman}, for every $g\in\mathcal D(\mu)$, we have
\[
\int_\D |g'(z)|^2 U_\mu(z)\,dA(z)=\int_{\overline{\D}} D_z(g)\,d\mu(z),
\]
where
\[
D_z(g)=\int_\T \left|\frac{g(z)-g(\zeta)}{z-\zeta}\right|^2\,dm(\zeta)
\]
and $m$ is the normalized Lebesgue measure on $\T$. Let $f\in\mathcal D(\mu)$. Using \eqref{eq:norme-Dmu}, we get that 
\begin{eqnarray}
\|\chi_n f\|_{\mathcal D(\mu)}^2&=&\|\chi_n f\|_2^2+\int_{\D}|(\chi_n f)'(z)|^2 U_\mu(z)\,dA(z)\notag\\
&=&\|f\|_2^2+\int_{\overline \D}D_z(\chi_n f)\,d\mu(z).\label{eq:cleH8-Dmu}
\end{eqnarray}
Observe now that 
\[
D_z(\chi_n f)=\int_\T \left|\frac{z^nf(z)-\zeta^nf(\zeta)}{z-\zeta}\right|^2\,dm(\zeta),
\]
and straightforward computations show that 
\[
\left|\frac{z^nf(z)-\zeta^nf(\zeta)}{z-\zeta}\right|^2\leq 2n^2|f(z)|^2+2\left|\frac{f(z)-f(\zeta)}{z-\zeta}\right|^2.
\]
We then deduce that 
\[
D_z(\chi_n f)\leq 2 n^2 \|f\|_2^2+2 D_z(f),
\]
whence, according to \eqref{eq:cleH8-Dmu}, we obtain 
\begin{eqnarray*}
\|\chi_n f\|_{\mathcal D(\mu)}^2 &\leq & \|f\|_2^2+2n^2 \mu(\overline{\D}) \|f\|_2^2 +2\int_{\overline{\D}}D_z(f)\,d\mu(z) \\
&\leq & 2 \|f\|_{\mathcal D(\mu)}^2+2n^2  \mu(\overline{\D})\|f\|_{\mathcal D(\mu)}^2\\
&\lesssim & n^2 \|f\|_{\mathcal D(\mu)}^2.
\end{eqnarray*}
Therefore, we deduce that $\|\chi_n\|_{\mathfrak M(\mathcal D(\mu))}\lesssim n$, which proves (H8).
\end{proof}

Using Theorem~\ref{theorem-1bis} and Corollary~\ref{corollaire-1-bis}, we can recover a partial version of a result of Richter--Sundberg \cite[Corollary 5.5]{RichterSundbergDirichlet} and Aleman \cite{habilitation-aleman}. 
\begin{theo}(Richter--Sundberg, Aleman)\label{thm-aleman-cyclicite}
Let $\mu$ be a finite positive measure on $\overline{\D}$, and let $f,g\in\mathfrak M(\mathcal D(\mu))$ and assume that $|g(z)|\leq |f(z)|$ for every $z\in\D$. 
\begin{enumerate}
\item[$(1)$] There exists $N\in\mathbb N^*$ such that 
\[
[g^N]_{\mathcal D(\mu)}\subset [f]_{\mathcal D(\mu)}.
\] 
\item[$(2)$] Moreover, if $g$ is cyclic for $S$ in $\mathcal D(\mu)$, then $f$ is also cyclic for $S$ in $\mathcal D(\mu)$. 
\end{enumerate}
\end{theo}
\begin{proof}
According to Lemma~\ref{Lemma-D-mu-hypothese}, we can apply Theorem~\ref{theorem-1bis} and Corollary~\ref{corollaire-1-bis} which immediately gives the result. 
\end{proof}
It should be noted that Theorem~\ref{thm-aleman-cyclicite} is proved in \cite{RichterSundbergDirichlet} with $N=1$ and with the weaker assumption that $f,g\in\mathcal D(\mu)$ but with a measure $\mu$ on $\T$. The case of a measure on the closed unit disc is obtained in \cite{habilitation-aleman} with $N=1$. In both papers, the result is obtained using radial approximations technics. 

As an application of Theorem~\ref{theorem-2-bis}, we get the following result.
\begin{theo}
Let $\mu$ be a finite positive measure on $\overline{\D}$ and let $f\in\mathfrak M(\mathcal D(\mu))\cap A(\D)$ be an outer function and assume that $\mathcal Z(f)=\{\zeta_0\}$ for some $\zeta_0\in\T$ which is not a bounded point evaluation of $\mathcal D(\mu)$. Then $f$ is cyclic for $S$ in $\mathcal D(\mu)$. 
\end{theo}
 
\begin{proof}
Since $\zeta_0$ is not a bounded point evaluation of $\mathcal D(\mu)$, Lemma~\ref{lem:point-spectrum-S-b-cyclicity} implies that $z-\zeta_0$ is cyclic for $S$ in $\mathcal D(\mu)$. Then, according to Lemma~\ref{Lemma-D-mu-hypothese}, we can apply  Theorem~\ref{theorem-2-bis} which gives the result. 
\end{proof}
It should be noted that in \cite{MR3475457} (in the case when $\mu$ is a measure on $\T$), O. El-Fallah, Y. Elmadani and K. Kellay proved that $\zeta$ is a bounded point evaluation of $\mathcal D(\mu)$ if and only if $c_\mu(\zeta)>0$, where $c_\mu$ is the Choquet capacity associated to $\mathcal D(\mu)$.  Moreover, they also showed that when $\mu$ has countable support, then a function $f\in\mathcal D(\mu)$ is a cyclic vector for the shift  precisely when $f$ is an outer function and $c_\mu(\mathcal Z(f))=0$. 
\bibliographystyle{plain}

\bibliography{sources.bib}

\end{document}